\documentclass[12pt]{article} 
\usepackage[utf8x]{inputenc}
\usepackage{tikz} 
\usepackage[all]{xy}
\usepackage{authblk}
\usepackage[babel=true]{csquotes}
\usepackage{amsmath,amsthm, amssymb,amsfonts}
\usepackage[hmargin=1in,vmargin=1in]{geometry}
\numberwithin{equation}{subsection} 
\xyoption{arc}
\usepackage{eucal}
\usepackage{indentfirst}
\usepackage{mathrsfs}
\usepackage{verbatim}
\usepackage{makeidx}
\usepackage{graphicx}
\usepackage{hyperref}
\usepackage{enumitem} 
\usepackage{extpfeil} 
\usepackage{dsfont}
\usepackage[T1]{fontenc}
\newtheorem{thm}{Theorem}[section]
\newtheorem{prop}[thm]{Proposition}
\newtheorem{lem}[thm]{Lemma}

\newtheorem{cor}[thm]{Corollary}

\newtheoremstyle{bidule}
{3pt}
{3pt}
{}
{}
{\scshape}
{.}
{.5em}
{}
\newtheorem{df}[thm]{Definition}
\theoremstyle{definition}
\newtheorem{rmk}{Remark}[section]

\newtheorem*{note}{Note}

\newtheorem*{warn}{Warning}
\newtheorem*{claim}{Claim}
\newtheorem{nota}{Notation}[section]

\newcommand{\C}{\mathcal{C}}
\newcommand{\Ca}{\mathcal{C}}
\newcommand{\F}{\mathcal{F}}

\newcommand{\K}{\mathcal{K}}

\newcommand{\D}{\mathcal{D}}
\newcommand{\Ba}{\mathcal{B}}
\newcommand{\Ga}{\mathcal{G}}

\newcommand{\Xa}{\mathcal{X}}

\newcommand{\M}{\mathscr{M}}
\newcommand{\V}{\mathbb{V}}
\newcommand{\J}{\mathcal{J}} 

\newcommand{\Ea}{\mathcal{E}} 

\newcommand{\Fa}{\mathcal{F}}
\newcommand{\Nv}{\mathscr{N}}

\renewcommand{\le}{\mathscr{L}}

\newcommand{\Rc}{\mathscr{R}}

\renewcommand{\P}{\mathscr{P}}

\newcommand{\Pcal}{\mathcal{P}}

\renewcommand{\to}{\longrightarrow}
\newcommand{\ol}{\overline}
\newcommand{\ul}{\underline}

\newcommand{\U}{\mathbb{U}}

\newcommand{\Ob}{\text{Ob}}
\newcommand{\n}{\textbf{n}} 
\newcommand{\m}{\textbf{m}} 
\newcommand{\q}{\textbf{q}} 
\newcommand{\p}{\textbf{p}} 
\newcommand{\0}{\textbf{0}} 
\renewcommand{\1}{\textbf{1}} 
\renewcommand{\2}{\textbf{2}} 
\newcommand{\3}{\textbf{3}} 

\newcommand{\tx}{\text}
\newcommand{\tld}{\widetilde}

\renewcommand{\to}{\longrightarrow}
\DeclareMathOperator\Id{Id}
\DeclareMathOperator\Hom{Hom}

\DeclareMathOperator\Set{\textbf{Set}} 
\DeclareMathOperator\Lax{Lax}

\DeclareMathOperator\Le{\mathcal{L}}
\DeclareMathOperator\cof{\mathbf{cof}}

\DeclareMathOperator\fib{\mathbf{fib}} 
\DeclareMathOperator\we{\mathbf{we}} 

\DeclareMathOperator\degb{\textbf{deg}} 
 
\DeclareMathOperator\lr{\mathbf{lr}} %
\DeclareMathOperator\oarc{\overrightarrow{\C}}
\DeclareMathOperator\oarcg{\overleftarrow{\C}}

\DeclareMathOperator\laxlatch{\mathbf{Latch}_{lax}}  
\DeclareMathOperator\colaxmatch{\mathbf{Match}_{colax}}  
\DeclareMathOperator\latch{\mathbf{Latch}}  
\DeclareMathOperator\colaxlatch{\mathbf{Latch}_{colax}}  
\DeclareMathOperator\Colax{Colax}
\newcommand{\ns}{\ul{n}} 
\newcommand{\ms}{\ul{m}} 
\newcommand{\px}{\P_{\ol{X}}} 

\title{Colax Reedy diagrams}
\author{Hugo V. Bacard}
\affil{Western University}
\date{}
\begin{document}
\maketitle

\begin{abstract}
We  use a theory of colax Reedy diagrams to show that the category of Segal $\M$-precategories  with fixed set of objects has a model structure for a symmetric monoidal model category $\M=(\ul{M},\otimes,I)$. What is relevant here is when $\M$ is monoidal for a non-cartesian product. The model structure is of Reedy style and generalizes the Reedy model structure for classical Segal $\M$-precategories when $\M$ is monoidal for the cartesian product. The techniques we use also generalize the  Reedy model structure for classical Reedy diagrams.
\end{abstract}
\setcounter{tocdepth}{1}
\tableofcontents

\section{Introduction}
This paper is a first step toward the existence of model structure for Segal enriched categories when the base of enrichment $\M=(\ul{M},\otimes,I)$ is monoidal for a noncartesian product. The motivating examples are what we should call Segal DG-categories, which are Segal like enriched categories over chain complexes. To define such (weakly) enriched categories one uses the so called colax (or oplax) diagrams. Their definition is a generalization of the notion of \emph{up-to-homotopy monoid} introduced by Leinster \cite{Lei3}.\\

Our interest in studying these structures was motivated in part by a project of Toën \cite{Toen_Tannaka} who imagined a theory of higher linear categories that would be used in his program of \emph{higher tannakian duality}. Another motivation come from Kock and Toën \cite{Kock-Toen} who outlined that a good homotopy theory of Segal linear categories could bring a conceptual proof of the Deligne conjecture. That direction has been considered by Shoikhet \cite{Shoikhet_Deligne}. \\

Unfortunately we don't have a good homotopy theory of these weakly enriched categories; even in the one-object case which corresponds to Leinster's monoids. By `good homotopy' we mean a model structure on all colax diagrams (= Segal precategories) such that fibrant objects satisfy the Segal conditions.\\ 
On top of that, having a nice model structure on Segal precategories doesn't seem to be straightforward. And in this paper we only give a little step ahead with the existence of a model structure of unital Segal precategories when we fix the set of objects. \\ 

The major problem comes from the Segal maps e.g  $\C(A,B,C) \to \C(A,B) \otimes \C(B,C)$, in which the tensor product $\otimes$ is on the `wrong side'. And because of this many operations that exists naturally when $\otimes =\times$ are no longer immediate. For example computing limits in the categories of Segal precategories is a bit technical (see Section \ref{limit_colax}). 

We use a language of \emph{locally Reedy $2$-category} to get most of the results in this paper. And this is not simply a `general nonsense' approach, because working in these settings covers many situations. For example \emph{Leinster's $n$-algebras} considered by  
Shoikhet \cite{Shoikhet_Deligne} are special case of colax diagrams indexed by locally Reedy $2$-categories. 

 Below we outline very briefly the content of the paper and some of the missing tools for having a homotopy theory.

\subsection{What is done here} 
We show that if we fixe a set $X$, then the category $\Pcal\C(X,\M)$ of \textbf{unital Segal precategories} has a Reedy style model structure (Theorem \ref{model-for-unital}). We give this result in a general context of normal colax diagrams indexed by a strict $2$-category $\C$ which is locally Reedy and \emph{direct-divisible} (Definition \ref{direct-divisible}). The later property gives us a control on diagrams indexed by such $2$-categories. In particular it gives as a canonical map from the \emph{colax latching space} to the \emph{colax matching space} for a truncated diagram. This map is natural and obvious for classical Reedy $1$-diagrams but in the colax situation it's no longer guaranteed for general locally Reedy $2$-categories. In addition to that being direct-divisible allows constructing inductively colax diagrams.

\subsection{What is missing}
As $X$ runs through the category $\Set$ of sets, we have a canonical fibred category of all unital Segal precategories: 
$$p: \Pcal\C(\M) \to \Set$$
where the fiber of $X \in \Set$ is of course $\Pcal\C(X,\M)$.  The major question is to determine what kind of \emph{reasonable homotopy theory} we can put on $p$. The fibred category $p$ is a fibration in monoidal categories in the sense that each fiber carries a monoidal structure. But it's also a \emph{monoidal fibration} in the sense of Shulman \cite{Shulman_monoidal_fib} with the appropriate tensor product (which can be found in \cite{SEC1}).\\

The following facts are not known yet.
\begin{enumerate}
\item We don't know for the moment if $\Pcal\C(\M)$ is complete and cocomplete even if each fiber is complete and cocomplete. In fact pushing forward a colax diagram $\F \in \Colax(\C,\M)$ along a strict $2$-functor $\gamma : \C \to \D$ seems to be complicate and pulling back colax diagrams doesn't preserve necessarily limits.
\item We don't know either if there exists a \emph{Segalification functor} that takes a colax diagram to another that satisfies the Segal conditions.
\item The absence of the previous functor makes it hard to determine what kind of \emph{reasonable weak equivalences} we can have on $\Pcal\C(\M)$.
\end{enumerate}

\section{Preliminaries}

Let $\M$ be a biclosed $2$-category which is locally complete and cocomplete. By colax diagram in $\M$ we mean a colax morphism of $2$-categories $\F: \C \to \M$ where $\C$ is a strict $2$-category. We will denote by $\Colax(\C,\M)$ the category of colax morphisms and transformations which are icons in the sense of Lack \cite{Lack_icons}. Similarly a lax diagram in $\M$ is a lax morphism indexed by a strict $2$-category $\C$; we have a category $\Lax(\C,\M)$ of lax morphisms and icons.
\begin{nota}
We will denote by:
\begin{itemize}[label=$-$]
\item $\Colax(\C,\M)_n$ the full subcategory $\subset \Colax(\C,\M)$, of \emph{normal colax functors}. These are colax functors $\Fa$ such that the maps `$\Fa(\Id) \to \Id$' are identities and all the colaxity maps $ \Fa( \Id \otimes f) \to \Fa(\Id) \otimes \Fa(f)$ are natural isomorphisms.
\item $\Lax(\C,\M)_n$ the subcategory  $\subset \Lax(\C,\M)$ of \emph{normal lax functors}.
\end{itemize} 
\end{nota}

\subsection{Lax to Colax and vice versa} An easy exercise shows that we have an isomorphism of $1$-categories:
$$ \Colax(\C,\M) \xrightarrow{\cong} \{\Lax(\C^{2\tx{-op}},\M^{2\tx{-op}})\}^{op}$$ where 
the superscript `$2$-op' represents the $2$-opposite construction: we keep the same $1$-morphisms and reverse the $2$-morphisms. The isomorphism takes a colax morphism given by $$\F=\{ \F_{xy}: \C(x,y) \to \M(\F x,\F y); \varphi: \F(f\otimes g) \to \F(f) \otimes \F(g) \}$$ to the lax morphism given by $$\F^{2\tx{-op}}=\{ \F_{xy}^{op}: \C(x,y)^{op} \to \M(\F x,\F y)^{op}; \varphi^{op}: \F(f) \otimes \F(g) \to  \F(f\otimes g) \}.$$

\section{Colax Reedy diagrams} A \emph{colax Reedy diagram} is an object of $\Colax(\C,\M)_n$, that is a normal colax morphism
$$\F: \C \to \M$$
where $\C$ is a locally Reedy $2$-category in the sense of  \cite[Def 6.1]{COSEC1}.  From now $\C$ will be a locally Reedy $2$-category (henceforth $\lr$-category) which is \emph{simple} in the sense of \cite[Def 6.4]{COSEC1}. Here `simple' means that we have a global linear extension $\degb$ for $1$-morphisms such that $\degb(f \otimes g)= \degb(f)+\degb(g)$. \ \\

Note that if $\C$ is an $\lr$-category then so is $\C^{2\tx{-op}}$;  and if moreover $\C$ is simple then so is $\C^{2\tx{-op}}$. Given a $2$-morphism $z$ in some $\C(A,B)$ we introduced the  notion of lax-latching category at $z$ (\cite[Def 6.1]{COSEC1});  and of lax-latching object $\laxlatch(\F,z)$ for a lax diagram $\F$ (\cite[Def 6.10]{COSEC1}). 

\begin{df}
Let $\F: \C \to \M$ be a colax Reedy diagram in $\M$ and $z$ an $1$-morphism of $\C$ in some $\C(A,B)$. 
\begin{enumerate}
\item Define the \textbf{colax-matching category} at $z$, denoted $\partial^{\bullet}_{z/ \C}$, to be \textbf{the opposite category of the lax-latching category} of $\C^{2\tx{-op}}$ at $z$.
\item Define the \textbf{colax-matching object} of $\F$ at $z$ to be the lax-latching object of $\F^{2\tx{-op}}$ at $z$ i.e
$$\colaxmatch(\F,z):= \laxlatch(\F^{2\tx{-op}},z).$$
\item Define the \textbf{colax-latching category} at $z$ to be the usual latching category at $z$ of the Reedy $1$-category $\C(A,B)$ at $z$.
\item Define the \textbf{colax-latching object} of $\F$ at $z$ to be  the classical latching object of $\F_{AB}$  at $z$:
$$\colaxlatch(\F,z):= \latch(\F_{AB},z).$$
\end{enumerate}
\end{df}
We give below examples of morphisms in the colax-matching category. Let $z$ be a $1$-morphism of $\C$ and $(x_1,x_2)$ be a pair of composable $1$-morphisms such that $\otimes(x_1,x_2)= z$. Let $(y_1,y_2,y_3)$ be a triple of composable $1$-morphisms. Assume furthermore that we have two inverse $2$-morphisms $u_1: x_1 \to y_1$ and $u_2: x_2 \to y_2 \otimes y_3$. Then 

\begin{itemize}[label=$-$]
\item the maps $\beta_1 = \Id_z: z \to  \otimes(x_1,x_2)$ and $\beta_2= u_1 \otimes u_2:  z \to \otimes(y_1,y_2,y_3)$ are two objects of $\partial^{\bullet}_{z/ \C}$

\item the following diagram represents a morphism 
$u:\beta_1 \to \beta_2$ in the colax-matching category:

\[
\xy
(0,15)*+{(x_1,x_2)}="A";
(40,0)*+{(y_1,y_2,y_3)}="B";
(0,0)*+{(y_1,y_2 \otimes y_3)}="C";
{\ar@{.>}^-{u}"A";"B"};
{\ar@{->}^-{(u_1,u_2)}"A";"C"};
{\ar@{->}_-{\tx{co-composition}}"C";"B"};
\endxy
\]
\end{itemize}
Here ``co-composition'' represents the opposite (morphism of) composition.
\begin{rmk}\ \
\begin{enumerate}
\item The lax-latching object $\laxlatch(\F^{2\tx{-op}},z)$ is computed as colimit in $\M(\F A,\F B)^{op}$ which means that it's a limit in $\M(\F A,\F B)$. 
\item On can easily check that we  have two universal maps: 
$$\F z \to \colaxmatch(\F,z) \ \ \ \ \  \tx{and} \ \ \ \ \  \colaxlatch(\F,z) \to \F z$$ the second map being the usual morphism for the Reedy diagram $\F_{AB}$. Their composite gives a unique map 
$$i_z:\colaxlatch(\F,z) \to \colaxmatch(\F,z).$$
\item As pointed out in \cite[sec 6.1]{COSEC1} any classical Reedy $1$-category  $\Ba$ can be considered as a simple locally Reedy $2$-category (denoted $\Ba_{0 \to 1}$). Normal lax functor $\F: \Ba_{0 \to 1} \to \M$ are the same thing as normal colax functor  $\F: \Ba_{0 \to 1} \to \M$ and both of them are equivalent to $1$-functors from \footnote{Actually the cateogory of functor from $\Ba$ to $\ul{M}$} $\Ba$ to $\M$. We leave the reader to check that the objects defined in the definition coincide with the classical ones for $\Ba$ and $z \in \Ba$.
\end{enumerate}
\end{rmk}

\subsection{Some restrictions: direct-divisibility}
 It's important to notice that in the absence of $\F z$, there is no reason to have by universal property the previous morphism $i_z: \colaxlatch(\F,z) \to \colaxmatch(\F,z)$! Indeed $\colaxlatch(\F,z)$ depends only on the values of $\F_{AB}$ while $\colaxmatch(\F,z)$ depends also on the values of $\F(s) \otimes \F(t)$ for all $s,t$ with  an inverse morphism $z \to s \otimes t$. 

In particular if we consider a notion of truncation $\F^{\leq m}$ of $\F$ and if $z$ is of degree $m+1$ then unlike the classical case, we cannot produce a map $i_z:\colaxlatch(\F,z) \to \colaxmatch(\F,z)$. And this map is needed in the inductive method that provides the factorization axiom in the Reedy model structure (see \cite{Hov-model}, \cite{Jardine-Goerss}, \cite{Hirsch-model-loc}, \cite{Lurie_HTT}).

As we wish to use the same method for classical Reedy diagrams we need to guarantee the existence of that map. This leads us to some restrictions which happens to cover our known cases. The restriction is on the indexing $2$-category $\C$. 

\begin{nota}
Let $\C$ be a locally Reedy $2$-category which is simple. 
\begin{itemize}[label=$-$]
\item We will denote by $\oarc$ the $2$-subcategory of $\C$ consisting of all direct $2$-morphisms: $\Ob(\oarc)= \Ob(\C)$ and $\oarc(A,B):= \overrightarrow{\C(A,B)}$.
\item Similarly we will denote by $\oarcg$ the $2$-subcategory of $\C$ consisting of all inverse $2$-morphisms: $\oarcg(A,B):= \overleftarrow{\C(A,B)}$.
\end{itemize}
\end{nota}

\begin{df}\label{direct-divisible}
Let $\C$ be a simple locally Reedy $2$-category. Say that $\C$ is \textbf{direct-divisible} if for every triple $(A,B,C)$ of objects of $\C$ the composition functor on $\oarc$ 
$$c: \oarc(A,B) \times \oarc(B,C) \to \oarc(A,C)$$
is a Grothendieck fibration. 
\end{df}

The definition says that for any \ul{direct} $2$-morphism $\alpha: z \to z'$ of $\C$ in $\C(A,C)$ then for any $(s',t') \in \C(A,B) \times \C(B,C)$ such that $s' \otimes  t'=z'$; there exists  a \textbf{unique} pair $(s,t)$ such that $s \otimes t= z$ and two  unique direct maps $\beta_1: s \to s', \beta_2: t \to t'$ such that $\beta_1 \otimes \beta_2= \alpha$. The uniqueness comes from two facts: that in a classical Reedy category the only isomorphisms are identities;  and that $\C$ is simple.\ \\

Examples of such $2$-categories  include $(\Delta^+,+,\0)$, $\Ba_{0 \to 1}$, $\P_{\D}, \P_{\ol{X}}$ and their respective $1$-opposite. Indeed for $(\Delta^+,+,\0)$, one has that for every monomorphism $f: \n \to \m$ and for every $\p, \q$ such that we have a side-by-side decomposition $\p + \q= \m$; taking the inverse image of each side gives two monomorphisms $f_1: \n_1 \to \p$ and $f_2: \n_2 \to q$ such that $f_1 + f_2= f$. 

\subsection{Truncation of (colax) diagrams} 

\paragraph{$2$-groupement or almost-$2$-category} In \cite{COSEC1} we consider a notion of \emph{$2$-groupement} following Bonin \cite{Bonnin}. The idea is that when we consider a simple $\lr$-category $\C$ with the same objects but only $1$-morphisms of degree $\leq m$, we no longer have a $2$-category but an `almost $2$-category'; in the sense that the composition of $f$ and $g$ is defined only if $\degb(f) + \degb(g) \leq m$ (see \cite{COSEC1} for  details). So morally a $2$-groupement is a sort of $2$-category where the composition is not automatic but is subjected to a condition. \\
\ \\
From now we will denote by $\C^{\leq m}$ the $2$-groupement (or almost-$2$-category) given by:
\begin{itemize}[label=$-$]
\item the same objects as $\C$;
\item $\C^{\leq m}(A,B):= \C(A,B)^{\leq m}$ the full subcategory  of $\C(A,B)$ of $1$-morphisms of degree $\leq m$;
\item The composition is partially defined but is associative. 
\end{itemize} 
We have a corresponding notion of  (co)lax morphisms between $2$-groupement and transformations. Furthermore many notions that exist for $\C$ restricts naturally to $\C^{\leq m}$. \ \\

A  \textbf{truncated colax diagram} is a colax morphism $\Ga: \C^{\leq m} \to \M$ for some $\lr$-category $\C$ and an ordinal $m$. In particular any colax diagram $\F: \C \to \M$ induces a truncated diagram $\F^{\leq m}: \C^{\leq m} \to \M$ for all $m$.\ \\

One of the advantages of having a direct divisible $\lr$-category $\C$ is that we can establish the:

\begin{lem}\label{latch-match}
Let $\C$ be a locally Reedy $2$-category which is simple and direct-divisible; and $z$ be a $1$-morphism of degree $m$. Then for any colax diagram $\F: \C^{\leq m-1} \to \M$ there is a canonical map:
$i_z:\colaxlatch(\F,z) \to \colaxmatch(\F,z)$.
\end{lem}

\subsection*{Proof of Lemma \ref{latch-match}}
We outline the general idea of the proof leaving the details to the reader.\ \\

Since $\colaxlatch(\F,z)$ is a colimit, the map $i_z$ will follow by universal property of the colimit  if we show that for any direct $2$-morphism 
$\alpha: s \to z$ with $\degb(s)\leq m-1$, we have a map $i_z(s): \F s \to \colaxmatch(\F,z)$ which is functorial in $\alpha$. \\ 

Recall that  $\colaxmatch(\F,z)$ is the limit of $\otimes( \F x_1, ..., \F x_n)$  where the limit is taking over an appropriate category of inverse $2$-morphisms  $\beta: z \to  \otimes(x_i)$. It follows that the map 
$$i_z(s): \F s \to \colaxmatch(\F,z)$$ 
will be induced by universal property of the limit if we show that we have a compatible diagram $\F s \to  \otimes(\F x_i)$ for any inverse $2$-morphism $\beta: z \to  \otimes(x_i)$ in the colax-matching category of $z$. 
Below we outline how we get the map $\F s \to  \otimes( \F x_1, ..., \F x_n)$.
\begin{enumerate}
\item If $\beta: z \to  \otimes(x_i)$ is not the identity, then $\degb(x_1)+ ... + \degb(x_n) < m$, and $\otimes(x_1,...,x_n)$ is a $1$-morphism of $\C^{\leq m-1}$. So given $\alpha: s \to z$ and such $\beta$, the composite 
$\beta \circ \alpha: s \to \otimes(x_1,...,x_n)$ is a $2$-morphism of $\C^{\leq m-1}$ where $\F$ is defined; and we get a morphism: 
$$\F(\beta \circ \alpha): \F s \to \F [\otimes(x_1,...,x_n)].$$ 

Now using the colaxity maps of $\F$ (and the coherence) we have a map:
$$\varphi: \F [\otimes(x_1,...,x_n)] \to \otimes( \F x_1, ..., \F x_n)$$

The composite of the two previous maps gives:
$$\F s \xrightarrow{\varphi \circ \F(\beta \circ \alpha)} \otimes( \F x_1, ..., \F x_n).$$
\item If $\beta$ is the identity i.e $\otimes(x_1,...,x_n)=z$ this is where we use the fact that the composition in $\oarc$  is a Grothendieck fibration (the \textbf{direct-divisibility}). As $\alpha : s \to z$ is a direct $2$-morphism and since the composition of direct $2$-morphisms is a fibration then for any such $(x_1,...,x_n)$ we can find a unique $n$-tuple of direct $2$-morphisms $\alpha_i: s_i \to x_i$ such that: 
\begin{itemize}[label=$-$]
\item $\otimes(s_1,...,s_n)=s $ and 
\item  $\otimes( \alpha_1, ..., \alpha_n)= \alpha$.  
\end{itemize}
We define  the map $\F s \to \otimes( \F x_1, ..., \F x_n)$ to be the composite:
$$\F s \xrightarrow{\varphi} \otimes( \F s_1, ..., \F s_n) \xrightarrow{\otimes(\F \alpha_i)} \otimes( \F x_1, ..., \F x_n). $$
\end{enumerate}

\paragraph{Compatibility} It remains to show that these data are compatible, in the sense that if we have $\beta_1: z \to \otimes(x_1,...,x_n)$, $\beta_2: z \to \otimes(y_1,...,y_p)$ with a map $u: \beta_1 \to \beta_2$ in the colax-matching category at $z$ (whence $n\leq p$); then we must have a commutative diagram:
\[
\xy
(0,15)*+{\F s}="A";
(40,15)*+{\otimes( \F x_1, ..., \F x_n)}="B";
(0,0)*+{\otimes( \F y_1, ..., \F y_p)}="C";
{\ar@{->}^-{}"A";"B"};
{\ar@{->}^-{}"A";"C"};
{\ar@{->}^-{\F u}"B";"C"};
\endxy
\]

First of all if $\beta_1: z \to \otimes(x_1,...,x_n)$ is not the identity, i.e $\degb(x_1)+ ... + \degb(x_n) < m$, then the commutativity is given by the dual statement of \cite[Prop. 6.9]{COSEC1} for colax diagrams. That proposition says that any truncated colax diagram $\F: \C^{\leq m} \to \M$ induces a functor on the colax-matching category at $z$ for every $z$ of degree $m+1$.\ \\

So we can assume that $\otimes(x_1,...,x_n)= z$. We consider the two cases: when $p>n$ and when $p=n$. 

\paragraph{\ul{Case 1: $p=n$}} In this case the map $u$ is given by an $n$-tuple of inverse $2$-morphisms i.e $u=(u_1,..., u_n)$ with $u_i: x_i \to y_i$;  moreover we have that $\otimes(u_1, ..., u_n)= \beta_2 $ by definition of the colax-matching category at $z$. We can assume that $u$ is not the identity (otherwise it's trivial); it follows that $\beta_2$ is not the identity because the composition in a simple $\lr$-category is identity reflecting\footnote{If $\beta_2$ is the identity, by reflection and induction all the $u_i$ are identities which contradicts the assumption ``$u$ is not the identity''}.  Since  $\beta_2$ is not the identity, we have by definition, that the map $\F s \to \otimes( \F y_1, ..., \F y_n)$ is the composite:
$$\F s \xrightarrow{\F(\beta_2 \circ \alpha)}  \F[\otimes(y_1,...,y_n)]  \xrightarrow{\varphi}\otimes( \F y_1, ..., \F y_n).$$  

Now since  $\alpha= \otimes(\alpha_i)$ and $\beta_2= \otimes (u_i)$ we have: 
$$ \beta_2 \circ \alpha = \otimes (u_i) \circ \otimes(\alpha_i)= \otimes( u_i \circ \alpha_i).$$

Using the functoriality of the coherence of $\F$ we have a diagram where everything commutes:
\[
\xy
(0,20)*+{\F s}="A";
(40,20)*+{\otimes( \F s_1, ..., \F s_n)}="B";
(0,0)*+{\F[\otimes(y_i)]}="C";
(80,10)*+{\otimes( \F x_1, ..., \F x_n)}="M";
(40,0)*+{\otimes( \F y_1, ..., \F y_n)}="D";
{\ar@{->}^-{\varphi}"A";"B"};
{\ar@{->}_-{\F(\beta_2 \circ \alpha)=\F[\otimes(u_i \circ \alpha_i)]}"A";"C"};
{\ar@{->}^-{\otimes(\F \alpha_i)}"B";"M"};
{\ar@{->}^-{\otimes(\F u_i)}"M";"D"};
{\ar@{->}_-{\otimes[\F (u_i \circ \alpha_i)]}"B";"D"};
{\ar@{->}^-{\varphi}"C";"D"};
\endxy
\]  

That diagram gives us the required compatibility. 

\paragraph{\ul{Case 2: $p>n$}} For this case, an analysis of the colax-matching category at $z$ tells us that the morphism $u$ is given  by:

$$ (x_1,..., x_n) \xrightarrow{(u_1',...,u'_n)}  (y_1',..., y_n')  \xrightarrow[\sigma^{op}]{\tx{some co-compositions}}(y_1,...,y_p).$$

Here $y_k'$ is a composite of some the $y_i$'s  i.e  $y_k'= \otimes(y_l,.., y_{l+j})$; and the maps $u_k': x_k \to y_k'$ are inverse maps. The map $\sigma$ is governed by a map of $\Delta^+$ which is surjective. Below we show an example of such a map $u$ for $n=2$ and $p=3$:
\[
\xy
(0,20)*+{(x_1,x_2)}="A";
(40,0)*+{(y_1,y_2,y_3)}="B";
(0,0)*+{\underbrace{(y_1',y_2')}_{=(y_1,y_2 \otimes y_3)}}="C";
{\ar@{.>}^-{u}"A";"B"};
{\ar@{->}^-{(u_1',u_2')}"A";"C"};
{\ar@{->}_-{\sigma^{op}}"C";"B"};
\endxy
\]

In the above example the map $\sigma$ is governed by the map 
$\sigma_3^2: \3 \to \2$ of $\Delta^+$ given by $\sigma_3^2(1)= 1$ and $\sigma_3^2(2)= \sigma_3^2(3)=2$. \ \\

By the previous case ($p=n$) we know that the following commutes:  
\[
\xy
(0,15)*+{\F s}="A";
(40,15)*+{\otimes( \F x_1, ..., \F x_n)}="B";
(0,0)*+{\otimes( \F y_1', ..., \F y_n')}="C";
{\ar@{->}^-{}"A";"B"};
{\ar@{->}^-{}"A";"C"};
{\ar@{->}^-{\F u'}"B";"C"};
\endxy
\]
thus our original diagram will be commutative if we show that the following one is also commutative:
\[
\xy
(0,15)*+{\F s}="A";
(40,15)*+{\otimes(  \F y_1', ..., \F y_n')}="B";
(0,0)*+{\otimes(\F y_1, ..., \F y_p)}="C";
{\ar@{->}^-{}"A";"B"};
{\ar@{->}^-{}"A";"C"};
{\ar@{->}^-{\F \sigma}"B";"C"};
\endxy
\]

So we are reduced to the case where $u=\sigma^{op}$ is  a map that composes some of the $x_i$'s. As mentioned above such a map $\sigma$ is governed by a map of $\Delta^+$ which is surjective (the matching category is a subcategory of the Grothendieck integral of a functor defined over $\Delta^+$). By a theorem of Mac Lane \cite{Mac}  we know that all surjective maps of $\Delta^+$ are generated by the cofaces $\sigma^i$ in the sense that any surjective map $\sigma$ can be written as a composite of some $\sigma^i$. From that observation it's not hard to see that any such map $\sigma:  (y_1',..., y'_n)\to  (y_1,..., y_p)$  in the colax category can be written as a composite of maps governed by the cofaces $\sigma^i$ (by definition of the Grothendieck construction). 
\ \\

But since $\F$ defines a functor on the colax-matching category (thanks to the coherence conditions), it's enough to show that we have a commutative diagram when $\sigma$ is governed by single map $\sigma^i$ (whence $p= n+1$). Such a map $\sigma:  (y_1',..., y'_n)\to  (y_1,..., y_{n+1})$ consists of  composing $y_i$ and $y_{i+1}$ ($0 \leq i \leq n-1$). We remind the reader that the map $\F \sigma$ is essentially given by the colaxity map $\varphi: \F(y_i \otimes y_{i+1}) \to \F y_i \otimes \F y_{i+1}$:
$$ \otimes(  \F y_1', ..., \F y_n') \xrightarrow{\otimes (\Id, ..., \varphi, ..., \Id)} \otimes(\F y_1, ..., \F y_{n+1}) .$$ 
After these reductions, we simply have to observe that if we have an $n$-tuple of directs maps $(\alpha_l : s_l \to y_l')_{(1 \leq l \leq n)}$ such that
 $\alpha= \otimes(\alpha_l)$; then if we apply the \textbf{direct-divisibility condition} to the map $\alpha_i: s_i \to (y_i\otimes y_{i+1})$ we can find two (unique) direct maps $\alpha_i': s_i' \to y_i$ and $\alpha_{i+1}': s_{i+1}' \to y_{i+1}$ such that $\alpha_i' \otimes \alpha_{i+1}'= \alpha_i$. On the one hand  we have, using the functoriality of the coherences for $\F$, a commutative diagram:
\[
\xy
(0,20)*+{\F y_i'=\F (y_i \otimes y_{i+1})}="A";
(40,20)*+{\F y_i \otimes \F y_{i+1}}="B";
(0,0)*+{\F s_i}="C";
(40,0)*+{\F s_i' \otimes \F s_{i+1}'}="D";
{\ar@{->}^-{\varphi}"A";"B"};
{\ar@{->}_-{\F \alpha_i }"C";"A"};
{\ar@{->}_-{\F \alpha_i' \otimes \F \alpha_{i+1}'}"D";"B"};
{\ar@{->}^-{\varphi}"C";"D"};
\endxy
\]  

One the other hand the direct-divisibility implies that the $(n+1)$-tuple of morphisms 
$$(\alpha_1,..., \alpha_{i}', \alpha_{i+1}',..., \alpha_l,...)$$
\textbf{is the one} we used to define the map 
$\F s \to \otimes (\F y_1, ..., \F y_{n+1})$. Now if we combine the two diagrams we get  the following commutative square:
\[
\xy
(0,20)*+{\otimes(  \F y_1', ..., \F y_n')}="A";
(60,20)*+{\otimes(\F y_1, ..., \F y_{n+1})}="B";
(0,0)*+{\otimes(\F s_1, ..., \F s_n)}="C";
(60,0)*+{\otimes(\F s_1, ..., \F s_i', \F s_{i+1}',... \F s_n)}="D";
{\ar@{->}^-{\sigma}"A";"B"};
{\ar@{->}_-{\otimes(\F \alpha_1, ..., \F \alpha_n)}"C";"A"};
{\ar@{->}_-{\otimes(\F \alpha_1, ...,\F \alpha_i', \F \alpha_{i+1}',... \F \alpha_n)}"D";"B"};
{\ar@{->}^-{\otimes (\Id, ..., \varphi, ..., \Id)}"C";"D"};
\endxy
\]  

We get our desired commutative diagram by precomposing with the colaxity map $\varphi : \F s \to \otimes(\F s_1, ..., \F s_n)$; this shows that we have a compatible diagram.\ \\

Therefore by universal property of the limit we have a unique map $i_z(s): \F s \to \colaxmatch(\F,z)$ that makes everything compatible.\ \\
\ \\
The fact that the map $i_z(s)$ is functorial in $\alpha: s \to z$ is tedious but straightforward. One has to use the fact that the direct-divisibility condition says that we have  fibrations between Reedy categories; and since in Reedy categories there are no non-trivial isomorphisms, then the inverse image functor commutes genuinely with the composition. Consequently the cartesian lifting of a composite is \ul{the} composite of the liftings. We leave the details to the reader.\ \\ 

Summing up our discussion we get our unique map 
$$i_z:\colaxlatch(\F,z) \to \colaxmatch(\F,z)$$ by universal property of the colimit; and the lemma follows.
\qed

\section{Colax diagram and simplicial objects}
\subsection*{Warning: $\Delta$ and $\Delta^+$}
Below we mention two different categories ``\textbf{Delta}''. We include this short paragraph to warn the reader about the potential confusion. We outline very briefly some known facts about these two ``Delta''.
\begin{itemize}[label=$-$]
\item $\Delta$ is the category of finite ordinals $\ns= \{0,..., n \}$,  \textbf{without the empty set}. The morphisms are the nondecreasing functions.
\item $\Delta^+$ is the category of all finite ordinals $\n= \{0,..., n-1\}$, \textbf{with the empty set} ($=\0$). The morphisms are also the nondecreasing functions. 
\end{itemize}

\subsubsection*{From $\Delta$ to $\Delta^+$}
If $\ns= \{0,..., n \}$ and $\ul{m}= \{0,..., m \}$  are two objects of $\Delta$, say that $f : \ns \to \ul{m}$ \emph{preserves the extremities} if:
$$f(0)= 0 \ \ \ \tx{and} \ \ \ f(n)=m. $$ 

Let $\Omega \subset \Delta$ be the subcategory having the same objects as $\Delta$ and whose morphisms are the ones that preserve the extremities.  Then we claim that:
\begin{claim}
There is an isomorphism of categories between $\Omega^{op}$ and $\Delta^+$.
\end{claim} 
\begin{note}
The above claim is known in the literature as an example of ``Joyal duality'' \cite{Joyal_theta}
\end{note}

We will not give a detailed proof of the claim but we will give the main idea. To show that the claim holds we explicitly construct an isomorphism $T: \Delta^+ \to \Omega^{op}$.\ \\

On the objects, $T$ maps $\n= \{0,..., n-1\}$ to $\ns= \{0,..., n \}$. To see what $T$ does on morphisms we need to go back to Mac Lane's description of the category $\Delta^+$ \cite[p.172]{Mac}. The category $\Delta^+$ has a monoidal structure given by the the ordinal addition $+$. Mac Lane showed that the arrows in $\Delta^+$ are generated by addition and composition from $\mu: \2 \to \1$ and $\eta: \0 \to \1$. 
Therefore in order to define $T$ we simply have to give $T(\mu)$ and $T(\eta)$. \\

The maps $T(\mu)$ and $T(\eta)$ are, respectively, the opposite of the following maps of $\Omega$:
\begin{itemize}[label=$-$]
\item  $T(\mu)^{op}: \{0,1\} \to \{0,1,2\}$, the unique map that takes $0$ to $0$ and $1$ to $2$. 
\item $T(\eta)^{op}: \{0,1\} \to \{0\}$, the unique constant map.
\end{itemize} 

We leave the reader to check that  the functor $T$ we get is an isomorphism.

\subsection{A result of Leinster: $\Delta_X$ and $\P_{\ol{X}}$} 
For a  set $X$, Bergner \cite{Bergner_rigid,Bergner_mon_seg} then Lurie \cite{Lurie_HTT} considered a category $\Delta_X$ that we described as follows (see also \cite[Ch. 10.1]{Simpson_HTHC}). 

\begin{enumerate}
\item The objects of $\Delta_X$ are sequences $(x_0,...,x_n)$ for $\ns \in \Delta$;
\item  A morphism $\phi: (x_0, ...,x_n) \to (y_0,..., y_m)$ is a morphism $\phi: \ns \to \ms$  of $\Delta$  with the property that:
$$ x_i = y_{\phi(i)}, \ \ \  i=0,...,n.$$
\item the composition is the obvious one. 
\end{enumerate}

It's easy to see that we have a fibred category $p : \Delta_X \to \Delta$; where $p(x_0,...,x_n)= n$ and $p(\phi)= \phi$. Indeed given a morphism $\phi: n \to m$ of $\Delta$ with $(y_0,..., y_m)$ over $m$ we have a cartesian lifting $\phi: (x_0,...,x_n) \to (y_0,..., y_m)$ if we take $x_i:= y_{\phi(i)}$, $i=0,...,n$. When $X$ has one element then $p : \Delta_X \to \Delta$ is an isomorphism. \ \\
\ \\
The category $\Delta_X$ has been used by Bergner \cite{Bergner_rigid,Bergner_mon_seg}, Lurie \cite{Lurie_HTT} and Simpson  \cite{Simpson_HTHC}, to define enriched categories having $X$ as set of objects as presheaves over $\Delta_X$ . That idea goes back to Grothendieck-Segal who observed that a small category $\Ba$ can be recovered by its nerve $\Nv(\Ba): \Delta^{op} \to \Set$. \ \\

For a set $X$ we've constructed in \cite{SEC1} a strict $2$-category $\P_{\ol{X}}$ which classifies lax-morphisms from $\ol{X}$ to a $2$-category $\M$ in the sense that we have an isomorphism of $1$-categories:
$$2\textbf{-Func}(\P_{\ol{X}}, \M) \cong \Lax(\ol{X}, \M)$$
functorial in $X$. We give a brief description of $\P_{\ol{X}}$ below.
\begin{enumerate}

\item The objects of $\P_{\ol{X}}$ are the elements of $X$;
\item a $1$-morphism from $x$ to $y$ is a sequence $(x_0,...,x_n)$ with $x_0=x$ and $x_n=y$ ;
\item the composition is the concatenation of chains;
\item the identity of $x$ is the chain $(x)$;
\item the $2$-morphisms are parametrized by the morphism of $\Delta^+$. In fact we have a Grothendieck opfibration $ \Le_{xy}: \P_{\ol{X}}(x,y) \to \Delta^+$ for each pair of elements. And these opfibrations organize to form a $2$-functor $\Le: \P_{\ol{X}} \to (\Delta^+, +, \0)$. For example we have the following $2$-morphisms which somehow generate all the other ones: 
\[
\xy
(0,0)*+{x}="X";
(10,10)*+{y}="Y";
(20,0)*+{z}="E";
{\ar@{->}^-{(x,y)}"X";"Y"};
{\ar@{->}^-{(y,z)}"Y";"E"};
{\ar@{->}_-{(x,z)}"X";"E"};
{\ar@{=>}"Y"+(0,-4);"X"+(10,2)};
\endxy
\ \ \ \ \ \ \ \ \ \ \ \ 
\xy
(0,10)*+{x}="X";
(10,0)*+{(x)}="Y";
(20,10)*+{x}="E";
{\ar@{->}^-{(x,x)}"X";"E"};
{\ar@{=>}"Y"+(0,4);"X"+(10,-2)};
\endxy
\]

In the above diagrams, the one on the left  is a $2$-morphism  $(x,y,z) \to (x,z)$ which is parametrized by the unique map $\sigma_0: \2 \to \1$ of $\Delta^+$; and the one on the right is a $2$-morphism $(x) \to (x,x)$ parametrized by the unique map $\0 \to 1$ of $\Delta^+$.

\end{enumerate}

\begin{rmk}\label{rmk-px-delta}
One can observe that the objects of $\Delta_X$ correspond exactly to the $1$-morphisms of $\P_{\ol{X}}$. The relationship between $\P_{\ol{X}}$ and $\Delta_X$ has been outlined by Leinster \cite{Lei2} when $X$ has a single object; we recall that relationship below.
\end{rmk}

\begin{prop}[Leinster]
If $\M=(\ul{M}, \times,1)$ is a monoidal category for the cartesian product, then we have an isomorphism of categories:
$$\Colax[(\Delta^+,+,\0), \M] \xrightarrow{\cong} \Hom(\Delta^{op}, \ul{M}).$$
\end{prop}
The result of Leinster has a general form given by:
\begin{prop}\label{Leinster_delta}
Let $\M=(\ul{M}, \times,1)$ be a monoidal category for the cartesian product. 
\begin{enumerate}
\item Then we have an isomorphism of categories:
$$\Colax[\P_{\ol{X}}, \M] \xrightarrow{\cong} \Hom(\Delta_X^{op}, \ul{M}).$$
\item The full subcategory $\Colax[\P_{\ol{X}}, \M]_n$ of normal colax morphisms, is isomorphic to the category of  unital Segal $\M$-precategories $\Pcal\C(X,\M)$.
\end{enumerate} 
\end{prop}
For the definition of Segal $\M$-precategories we refer the reader to \cite[Definition 10.1.1]{Simpson_HTHC}.

\subsubsection{Sketch of the proof of Proposition \ref{Leinster_delta}}
The idea of the proof is the same as the one given by Leinster.  We will only show how one constructs (functorially) a diagram $\Delta_X \to \ul{M}$ out of a colax diagram $\px \to \M$;  the inverse functor is obtained by reversing the process.\ \\

Let $\F: \px \to \M$ be a colax diagram. We will denote by $\Delta_{\F}$ the diagram we are about to construct.  As observed in Remark \ref{rmk-px-delta}, the objects of $\Delta_X$ are in one-one correspondence to the $1$-morphisms of $\px$; so it's clear how to define $\Delta_{\F}$ on objects. It remains to define $\Delta_{\F}$ on morphisms.\ \\

The morphisms of $\Delta_X$ are organized in two sets:
\begin{itemize}[label=$-$]
\item the set of morphisms $\phi: (x_0, ...,x_n) \to (y_0,..., y_m)$ such that $x_0=y_0$ and $x_n=y_m$ i.e, the ones such that $p(\phi)$ is a morphism of $\Omega$.
\item the set of morphisms $\phi: (x_0, ...,x_n) \to (y_0,..., y_m)$ such that $p(\phi)$ is not a morphism of $\Omega$.
\end{itemize}
We will define separately $\Delta_{\F}$ for morphisms over $\Omega$ and for the other morphisms.
\paragraph{Definition of $\Delta_{\F}$ for morphisms over $\Omega$.}
Let $\Omega_X$ be the category obtained by pullback of the fibration $p: \Delta_X \to \Delta$ along the inclusion $\Omega \to \Delta$:
\[
\xy
(0,15)*+{\Omega_X}="A";
(30,15)*+{\Delta_X}="B";
(0,0)*+{\Omega}="C";
(30,0)*+{\Delta}="D";
{\ar@{.>}^-{}"A";"B"};
{\ar@{.>}_-{q}"A";"C"};
{\ar@{->}_-{p}"B";"D"};
{\ar@{->}^-{}"C";"D"};
\endxy
\]  

In particular $q: \Omega_X \to \Omega$ is also a fibred category. Note that the fibred category $p: \Delta_X \to \Delta$ has the property that a map in $\Delta$ whose codomain is in the range of the image of $p$ has exactly a \textbf{unique lifting}. More precisely given $\phi: \ns \to \ms$ in $\Delta$ and $(y_0,..., y_m)$ over $\ms$, then there is exactly a unique map in $\Delta_X$ over $\phi$, whose codomain is  $(y_0,..., y_m)$. \ \\

Let $x$ and $y$ be two elements of $X$. The pair $(x,y)$ represents simultaneously:
\begin{itemize}[label=$-$]
\item an object $(x,y)$ of $\Delta_X$ and 
\item a $1$-morphism $(x,y)$ of $\px(x,y)$.
\end{itemize} 

As an object of $\Delta_X$, hence of $\Omega_X$, we can form the comma category 
$$\Omega(x,y):= {\Omega_X}_{/(x,y)}$$
 consisting of all morphisms of $\Omega_X$ with codomain $(x,y)$. We have a canonical fibred category
$$ {\Omega_X}_{/(x,y)} \to \Omega_X $$ 
which, composed with $q$ gives a fibred category:
$$\Omega(x,y) \to \Omega. $$ 
If we consider the opfibration $\Omega(x,y)^{op} \to \Omega^{op}$ then we have:
\begin{claim}
In the diagram:
\[
\xy
(0,15)*+{\px(x,y)}="A";
(30,15)*+{\Omega(x,y)^{op}}="B";
(0,0)*+{\Delta^{+}}="C";
(30,0)*+{\Omega^{op}}="D";
{\ar@{->}_-{\Le}"A";"C"};
{\ar@{->}_-{q}"B";"D"};
{\ar@{->}^-{\cong}"C";"D"};
\endxy
\]  
there is an isomorphism $J: \px(x,y) \to \Omega(x,y)^{op}$ that makes everything compatible.
\end{claim}
The definition of $J$ on objects is clear since an object in $\px(x,y)$ can be identified with a sequence $(x,...,y)$ i.e, an object of $\Omega(x,y)$. This correspondence is clearly a bijection. The two categories $\px(x,y) $ and $\Omega(x,y)^{op}$ share common properties, namely:
\begin{itemize}[label=$-$]
\item they are both Reedy categories;
\item they have a terminal object $(x,y)$ which is the \ul{unique} object over $\1$ (resp. $\ul{1}$);
\item morphisms that are over an identity morphism (in $\Delta^+$ or $\Omega$) are identities;
\item there are non nontrivial isomorphisms since they are Reedy categories.
\end{itemize} 
\ \\
It follows that the two opfibrations (or cofibred categories) $q$ and $\Le$ have the following properties:
\begin{itemize}[label=$-$]
\item each fiber is a discrete category i.e a set; 
\item for every morphism $\phi \in  \Delta^+$ (resp. $\Omega^{op}$), and for every fixed object $s$ over the codomain of $\phi$, there is a unique  lifting of $\phi$ with codomain $s$; and that lifting is (automatically) cocartesian.
\end{itemize}

Having these properties at hand, we can now define $J(\phi)$ for a morphism 
$$ \phi: (x_0 ..., x_n) \to (y_0 ...., y_m)$$
in $\px(x,y)$ over $\phi: \n \to \m$ in $\Delta^+$.  The map $\phi: \n \to \m \in \Delta^+$ corresponds to a unique map denoted again $\phi: \ns \to \ms$ in $\Omega^{op}$; and since $J((x_0 ..., x_n))$ is over $\ns$ there is a unique (cocartesian) lifting:
$$ J((x_0 ..., x_n)) \to s $$
of  $\phi: \ns \to \ms$ for some object $s$ over $\ms$. 
\begin{claim}
The object $s$ is precisely $J((y_0 ..., y_m))$. And we define $J(\phi)$ to be that lifting.
\end{claim}
To prove that the claim holds it's enough to assume that $\phi$ is over a codegeneracy $\sigma_i: \n \to \n+\1$ or a coface $d_i: \n+\1 \to \n$ of $\Delta^+$. Indeed, maps overs cofaces and codegeneracies generate (by composition) all other maps in $\px(x,y)$. This observation goes back to Mac Lane \cite{Mac}. Furthermore, without loss of generality, we can assume that $\n=\1$ or $\n=\0$ since the opfibrations $\Le_{xy}$ are compatible with the composition in $\px$ and the ordinal addition in $\Delta^+$.\ \\

After these reductions we have two cases to consider:
\begin{itemize}
\item $\phi: (x,t,y) \to (x,y)$  (over $\phi: \2 \to \1$)
\item $\phi: (x) \to (x,x)$  (over $\phi: \0 \to \1$ if $x=y$)
\end{itemize}  
for which it's clear that $s= J(x,y)$ (resp $s= J(x,x)$) since it's the unique object over $\ul{1}$ in $\Omega(x,y)^{op}$ (resp $\Omega(x,x)^{op}$). This shows that the previous claim holds so that $J(\phi)$ is defined. \ \\

The fact that $J$ is a functor and that has an inverse is straightforward; we leave it to the reader. \ \\

It's now clear how we define $\Delta_{\F} (\phi)$ for  $\phi$ over a morphism of $\Omega^{op}$. One has a decomposition:
$$\Omega^{op}_X \cong \coprod_{(x,y) \in X^2} \Omega(x,y)^{op} \cong \coprod_{(x,y) \in X^2} \px(x,y)$$
which gives:
$$\Hom(\Omega^{op}_X,\ul{M})  \cong \prod_{(x,y) \in X^2} \Hom(\px(x,y), \ul{M}).$$

\paragraph{Definition of $\Delta_{\F}$ for morphisms \ul{not} over $\Omega$.} Morphisms of $\Delta_X$ not over $\Omega$ are generated by composition by the non-inner cofaces and the maps over $\Omega$. So we can assume that these maps $\phi$ are of one of the following forms:
\begin{itemize}[label=$-$]
\item $\phi: (x_0,...,x_n) \to (x_0,...,x_n, x_{n+1})$ 
\item $\phi: (x_0,...,x_n) \to (y, x_0...,x_n)$. 
\end{itemize}

To define $\Delta_{\F}(\phi)$ for these two type of maps we use the colaxity maps of $\F$ as follows. For $\phi: (x_0,...,x_n) \to (x_0,...,x_n, x_{n+1})$ we get the map 
$$\Delta_{\F}(\phi): \F(x_0,...,x_n, x_{n+1}) \to \F(x_0,...,x_n)$$
as the composite
$$\F(x_0,...,x_{n+1})=\F[(x_0,...,x_n) \otimes(x_n, x_{n+1})] \xrightarrow{colax} \F(x_0,...,x_n) \times \F(x_n,x_{n+1}) \xrightarrow{pr_1} \F(x_0,...,x_n).$$ 

Similarly we get $\Delta_{\F}(\phi): \F(y, x_0,...,x_n) \to \F(x_0,...,x_n)$ as the composite:
$$\F(y,...,x_{n})=\F[(y,x_0) \otimes (x_0,...,x_n)] \xrightarrow{colax} \F(y,x_0) \times \F(x_0,...,x_n) \xrightarrow{pr_2} \F(x_0,...,x_n).$$ 

This complete the definition of $\Delta_{\F}$ on morphisms. It remains to show that $\Delta_{\F}$ is a functor  i.e, that it respects the composition and identities. This is tedious but not hard to check and we leave it to the reader. One mainly uses the fact that the colaxity maps are coherent; and that the coherence for $\F$ is functorial in the $2$-morphisms of $\px$. The later means that if $\alpha: s \to s'$ and $\beta: t \to t'$ are composable $2$-morphisms in $\px$ then the following commutes:
\[
\xy
(0,20)*+{\F(s\otimes t)}="A";
(30,20)*+{\F(s) \times \F(t)}="B";
(0,0)*+{\F(s' \otimes t')}="C";
(30,0)*+{\F(s') \times \F(t')}="D";
{\ar@{->}^-{colax}"A";"B"};
{\ar@{->}_-{\F(\alpha \otimes \beta) }"A";"C"};
{\ar@{->}^-{\F(\alpha) \times \F(\beta)}"B";"D"};
{\ar@{->}^-{colax}"C";"D"};
\endxy
\] 

This completes the proof of Assertion $(1)$. Assertion $(2)$ is left to the reader. $\qed$  
\section{Limits in $\Colax[\C, \M]_n$ } \label{limit_colax}
\begin{prop}
For any $2$-category $\C$ and any  locally cocomplete $2$-category $\M$ the category 
$\Colax(\C, \M)$ is cocomplete.
\end{prop}
\begin{proof}
Indeed colimits in the category of colax diagrams are computed level-wise.
\end{proof}
We leave the reader to check that `being a normal colax functor' is stable by colimits, therefore we have:
\begin{cor}\label{cor-colimit}
For any $2$-category $\C$ and any  locally cocomplete $2$-category $\M$ the category 
$\Colax(\C, \M)_n$ is cocomplete.
\end{cor}
For limits however things get complicated because limits are no longer computed level-wise.  This is the same thing as colimits in the category $\Lax(\C,\M)$.  In \cite{COSEC1} we show that under some conditions on $\M$, the category $\Lax(\C,\M)$ is cocomplete when $\M$ is so. The assumption on $\M$ is demanding that for every $1$-morphism $f$, the horizontal composition $f \otimes -$ commutes with colimits; something we know to be true if $\M$ is a biclosed $2$-category or a monoidal closed category.\ \\
Having that in mind, we might want to use the isomorphism mentioned earlier:
$$ \Colax(\C,\M) \xrightarrow{\cong} \{\Lax(\C^{2\tx{-op}},\M^{2\tx{-op}})\}^{op}$$
and say that limits in $\Colax(\C,\M)$ are the same as colimits in $\Lax(\C^{2\tx{-op}},\M^{2\tx{-op}})$. But the only problem with this is that $\M^{2\tx{-op}}$ may not satisfy the property `$f \otimes -$ commutes with colimits'; because this will be equivalent to ask that $f \otimes -$ commutes with limits in $\M$ ! And that assumption fails to be true in general.\ \\

So we cannot use our previous result for $\Lax(\C,\M)$. We give below a direct approach to compute limits in our specific case of $\Colax(\C, \M)_n$ where $\C$ is a locally Reedy $2$-category which is simple and direct-divisible. \ \\ 

\begin{nota}
If $m \in \lambda$ we will denote by $\tau_m : \Colax(\C^{\leq m+1}, \M)_n \to \Colax(\C^{\leq m}, \M)_n$ the restriction  functor induced by the inclusion $\iota : \C^{\leq m} \hookrightarrow \C^{\leq m+1}$. 
\end{nota}
The key step to compute limits is the following:
\begin{lem}\label{creation-limit}
The functor $\tau_m : \Colax(\C^{\leq m+1}, \M)_n \to \Colax(\C^{\leq m}, \M)_n$ creates limits.
\end{lem}
\subsubsection{Proof of Lemma \ref{creation-limit}} Let $\Xa: \J \to  \Colax(\C^{\leq m+1}, \M)_n$ be a diagram such that the composite $$\tau_m\Xa: \J \to  \Colax(\C^{\leq m}, \M)_n$$ has a limit $\Ea$. We construct below a diagram $\tld{\Ea} \in \Colax(\C^{\leq m+1}, \M)_n$ such that $\tau_m \tld{\Ea}= \Ea$ and $\tld{\Ea}$ is the limit of $\Xa$. To do so, we need to define $\tld{\Ea}(z)$ for every $1$-morphism of degree $m+1$ together with the colaxity maps out of $\tld{\Ea}(z)$.\\
By the results of the previous section we know that:
\begin{itemize}[label=$-$]
\item We have a canonical map $ i_z:\colaxlatch(\Ea,z) \to \colaxmatch(\Ea,z)$ (\textbf{this is where we need the direct-divisibility});
\item For every $i \in \Ob(\J)$  we have a factorization of the canonical map $i_z$:
$$ \colaxlatch(\Xa_i,z) \to  \Xa_i z \to \colaxmatch(\Xa_i,z) $$
\item The canonical projection $p_i:\Ea \to  \tau_m \Xa_i$ induces  a commutative square:
\[
\xy
(0,20)*+{\colaxmatch(\Ea,z)}="A";
(40,20)*+{\colaxmatch(\Xa_i,z)}="B";
(0,0)*+{\colaxlatch(\Ea,z)}="C";
(40,0)*+{\colaxlatch(\Xa_i,z)}="D";
{\ar@{->}^-{p_i}"A";"B"};
{\ar@{->}_-{i_z}"C";"A"};
{\ar@{->}_-{i_z}"D";"B"};
{\ar@{->}^-{p_i}"C";"D"};
\endxy
\] 

Note that there is an abuse of notation with the maps $p_i$ since they are not really component of $p_i$ but are induced by $p_i$.
\end{itemize}
Introduce $\Xa_{\infty} z= \lim_{\J} \Xa_i z$ and let $\pi_i: \Xa_{\infty} z \to \Xa_i z$ be the canonical map. For each $i \in \Ob(\J)$ we have a canonical map 
$\colaxlatch(\Ea,z) \to \Xa_i z$ given by the composite:
$$\colaxlatch(\Ea,z) \to  \colaxlatch(\Xa_i,z) \to  \Xa_i z $$
and it's not hard to see that these maps form a compatible diagram; thus there is a unique map $$\colaxlatch(\Ea,z) \to \Xa_{\infty} z $$
that makes everything compatible.\ \\

Let $Q_i(z)$ be the limit-object forming the pullback diagram:
\[
\xy
(0,20)*+{Q_i(z)}="A";
(40,20)*+{\Xa_{\infty} z}="B";
(0,0)*+{\colaxmatch(\Ea,z)}="C";
(40,0)*+{\colaxmatch(\Xa_i,z)}="D";
{\ar@{.>}^-{}"A";"B"};
{\ar@{.>}_-{}"A";"C"};
{\ar@{->}_-{}"B";"D"};
{\ar@{->}^-{p_i}"C";"D"};
\endxy
\]  

So basically we would write $Q_i(z)$ as a fiber product $ \colaxmatch(\Ea,z) \times_{\colaxmatch(\Xa_i,z)} \Xa_{\infty} z $. Here the map $\Xa_{\infty} z \to \colaxmatch(\Xa_i,z)$ is obviously the composite of $\pi_i: \Xa_{\infty} z \to \Xa_i z$ and the canonical map $\Xa_i z \to \colaxmatch(\Xa_i,z)$.\ \\

We leave the reader to check that we have a functor $Q(z): \J \to \M$ that takes $i$ to $Q_i(z)$ and we set:
$$ \tld{\Ea}z:= \lim_{\J}Q(z).$$ 

We have a canonical factorization of $i_{z}$ :
$$ i_z= \colaxlatch(\Ea,z)\to \tld{\Ea}z  \to \colaxmatch(\Ea,z) $$ 
which means in particular that $\tld{\Ea}z$ is equipped with a coherent family of colaxity maps .\ \\

Proceeding as previously for all $z$ of degree $m+1$ we get a colax diagram $\tld{\Ea}: \C^{\leq m+1} \to \M$. The reader can check that it satisfies the universal property of the limit. $\qed$ 
\ \\
\begin{cor}\label{cor-limits}
Let $\M$ be a locally complete $2$-category and $\C$ be a locally Reedy $2$-category which is simple and direct divisible. Assume furthermore that the degree function $\degb: \C \to \lambda$ has a minimal value $m_0$ for non identity $1$-morphisms. \\

Then the category $\Colax[\C, \M]_n$ has all small limits. 
\end{cor}

\begin{proof}[Sketch of proof]

It's not hard to observe that in the almost-$2$-category $\C^{\leq m_0}$, the only $1$-morphisms of degree $\leq m_0$ are identities, which are of  degree $0$; and the one of degree $m_0$. This is a consequence of being a minimal value. In addition to that, there are no nontrivial $2$-morphisms between $1$-morphisms of degree $m_0$; indeed, the factorization axiom for Reedy $1$-categories $\C(A,B)$ will contradict the minimality of $m_0$. 

Furthermore since we assumed that the composition in $\C$ adds the degree i.e, $\degb(x \otimes y)= \degb(x)+ \degb(y)$, it's easy to see that for $z$ such that $\degb(z)=m_0$, the only pairs $(x,y)$ such that $x \otimes y= z$ are: $(\Id,z)$ and $(z,\Id)$ (as $m_0$ is minimal).\ \\

If we put these observations together, we see that objects of the category $\Colax[\C^{\leq m_0}, \M]_n $ have no \textbf{pure colaxity maps} i.e, the only colaxity maps are the isomorphisms $\F(z) \to \Id \otimes \F(z)$ and $\F(z) \to \F(z) \otimes \Id$.  Consequently any object $\F \in \Colax[\C^{\leq m_0}, \M]_n $ is  determined by
 the \ family of functors 
 $$\{ \F_{AB}: \C^{\leq m_0}(A,B) \to  \M(\F A, \F B) \}_{(A,B) \in \Ob(\C)^2}.$$ 
 
Therefore limits in $\Colax[\C^{\leq m_0}, \M]_n $ are computed level-wise\footnote{We can say, in a fancy way, that $\Colax[\C^{\leq m_0}, \M]_n$ is equivalent to the category of these family of diagrams.}.
\ \\ 

Now since $\M$ is locally is complete we deduce that  $\Colax[\C^{\leq m_0}, \M]_n$ is complete. Applying inductively Proposition \ref{creation-limit} we see that $\Colax[\C, \M]_n$ is complete as well.
\end{proof}
\section{Lifting factorization systems}
The following is the analogue of what we did for lax diagrams in \cite[Section 6.1.2]{COSEC1}. All of the following considerations is a natural generalization of what is known for classical Reedy diagrams (see \cite{Hov-model}, \cite{Hirsch-model-loc} for example).\ \\

Let $\M$ be a $2$-category which is locally complete and cocomplete and such that each $\M(U,V)$ has a factorization system. For simplicity we will reduce our study to the case where $\M$ is a monoidal category having a factorization system $(L,R)$.  
\begin{nota}
Let $\C$ be as previously. We denote by:
\begin{itemize}[label=$-$]
\item $\Rc=$ the class of morphisms $\alpha: \Fa \to \Ga$ in $\Colax[\C, \M]_n$ such that for every $1$-morphism $z$, the map 
$$g_z: \Fa z \cup_{\colaxlatch(\Fa,z)} \colaxlatch(\Ga,z) \to \Ga z \in R ;$$
\item $\le=$ the class of morphisms $\alpha: \Fa \to \Ga$ such that for all $z$ the map $$\alpha_z: \Fa z \to  \Ga_z \times_{\colaxmatch(\Ga,z)} \colaxmatch(\F,z) \in L.$$
\item Similarly for each $m \in \lambda$ there are two classes $\le_m$ and $\Rc_m$ in $\Colax(\Ca^{\leq m},\M)$.   
\end{itemize} 
\end{nota}
By the same arguments as in the classical case, and as in \cite[6.1.2]{COSEC1} for lax diagrams, one can  prove the following:

\begin{prop}
\begin{enumerate}
\item Let $\alpha: \Fa \to \Ga$ be an object $\Colax[\C^{\leq m+1}, \M]_n $  such that $\tau_m\alpha$ has a factorization of type $(\le_m,\Rc_m)$:
$$\tau_m\Fa \xrightarrow{i} \K \xrightarrow{p} \tau_m\Ga.$$

Then there is a factorization of $\alpha$ of type $(\le_{m+1},\Rc_{m+1})$ in $\Colax[\C^{\leq m+1}, \M]_n $.
\item Let $\alpha: \Fa \to \Ga$ be in $\le_{m+1}$ (resp. $\Rc_{m+1}$). If $\tau_m \alpha$ has the LLP (resp. RLP) with respect to all maps in $\Rc_m$ (resp. $\le_m$) then $\alpha$ has the LLP (resp. RLP)  with respect to all maps in $\le_{m+1}$ (resp. $\Rc_{m+1}$).
\end{enumerate}
\end{prop}
\begin{proof}
Left to the reader.
\end{proof}

\begin{cor}\label{cor-factorize}
Under the above hypothesis, the pair $(\Rc, \le)$ is a factorization system on the category  $\Colax[\C, \M]_n $.
\end{cor}
\begin{proof}
The pair $(\le_{m_0},\Rc_{m_0})$ is factorization system on  $\Colax[\C^{\leq m_0}, \M]_n$. Apply inductively the previous corollary. 
\end{proof}
\section{The Reedy model structure}
Let $\M$ be a monoidal model category or a $2$-category which is locally a model category.\ \\

Say that a morphism $\sigma: \F \to \Ga$ in $\Colax[\C, \M]_n$ is:
\begin{itemize}[label=$-$]
\item a \textbf{weak equivalence} if for every $1$-morphism $z$, the component $\sigma_z: \F z \to \Ga z$ is a weak equivalence in $\M$;
\item a \textbf{Reedy cofibration} if for every $1$-morphism $z$ the following map is a cofibration in $\M$:
$$\Fa z \cup_{\colaxlatch(\Fa,z)} \colaxlatch(\Ga,z) \to \Ga z $$
\item a \textbf{Reedy fibration} if for every $1$-morphism $z$ the map:
$$\Fa z \to  \Ga_z \times_{\colaxmatch(\Ga,z)} \colaxmatch(\F,z)$$
is a fibration in $\M$. 
\end{itemize} 

The main result here is that:
\begin{thm}\label{model-reedy-diag}
Let $\C$ and $\M$ be as above. Then the following hold.
\begin{enumerate}
\item The three classes of weak equivalences, Reedy cofibrations and Reedy fibrations determine a model structure on $\Colax[\C, \M]_n$.
\item If $\C$ is a classical Reedy $1$-category viewed as a $2$-category with two objects, and $\M$ is a model category also viewed as $2$-category which is is locally a model category; then the model structure on  $\Colax[\C, \M]_n$ coincide with the classical model structure for the diagram category $\Hom(\C,\M)$.
\end{enumerate}
\end{thm}

\begin{proof}
The category $\Colax[\C, \M]_n$ is complete and cocomplete by Corollary \ref{cor-limits} and Corollary \ref{cor-colimit}. The three classes of maps are clearly closed under compositions and retracts. The class of weak equivalences satisfies the $3$-for-$2$ property.
In $\M$ we have two factorization systems $(\cof \cap \we; \fib)$ and $(\cof; \fib \cap \we)$;  each of them induces a factorization system on $\Colax[\C, \M]_n$ by Corollary \ref{cor-factorize}. This proves Assertion $(1)$. 
Assertion $(2)$ is elementary and is left to the reader.
\end{proof}
\subsection{Application: A model structure for unital Segal $\M$-precategories with fixed objects}
Let $X$ be a set and $\Pcal\C(X,\M)$ be the category $\Colax[\px, \M]_n$. The objects of  $\Pcal\C(X,\M)$ will be called \textbf{unital Segal precategories}. \ \\

If we apply the previous theorem for $\C=\px$ we get: 
\begin{thm}\label{model-for-unital}
For a set $X$ and $\M$ as above, the following hold.
\begin{enumerate}
\item The three classes of weak equivalences, Reedy cofibrations and Reedy fibrations determine a model structure on $\Colax[\px, \M]_n$.
\item If $\M=(\ul{M},\times , I)$ is monoidal for the cartesian product then the Reedy model structure on $\Colax[\px, \M]_n$ coincide with the Reedy model structure for unital diagrams in $\Hom(\Delta_X^{op}, \ul{M})$.
\end{enumerate}
\end{thm}

\bibliographystyle{plain}
\bibliography{Bibliography_These}

\begin{thebibliography}{10}

\bibitem{SEC1}
H.~V. {Bacard}.
\newblock {Segal Enriched Categories I}.
\newblock http://arxiv.org/abs/1009.3673.

\bibitem{COSEC1}
H.~V. {Bacard}.
\newblock {Lax Diagrams and Enrichment}.
\newblock June 2012.
\newblock http://arxiv.org/abs/1206.3704.

\bibitem{Bergner_rigid}
Julia~E. Bergner.
\newblock Rigidification of algebras over multi-sorted theories.
\newblock {\em Algebr. Geom. Topol.}, 6:1925--1955, 2006.

\bibitem{Bergner_mon_seg}
Julia~E. Bergner.
\newblock Simplicial monoids and {S}egal categories.
\newblock In {\em Categories in algebra, geometry and mathematical physics},
  volume 431 of {\em Contemp. Math.}, pages 59--83. Amer. Math. Soc.,
  Providence, RI, 2007.

\bibitem{Bonnin}
F.~{Bonnin}.
\newblock {Les groupements}.
\newblock http://arxiv.org/abs/math/0404233.

\bibitem{Jardine-Goerss}
Paul~G. Goerss and John~F. Jardine.
\newblock {\em Simplicial homotopy theory}, volume 174 of {\em Progress in
  Mathematics}.
\newblock Birkh\"auser Verlag, Basel, 1999.

\bibitem{Hirsch-model-loc}
Philip~S. Hirschhorn.
\newblock {\em Model categories and their localizations}, volume~99 of {\em
  Mathematical Surveys and Monographs}.
\newblock American Mathematical Society, Providence, RI, 2003.

\bibitem{Hov-model}
Mark Hovey.
\newblock {\em Model categories}, volume~63 of {\em Mathematical Surveys and
  Monographs}.
\newblock American Mathematical Society, Providence, RI, 1999.

\bibitem{Kock-Toen}
Joachim Kock and Bertrand To{\"e}n.
\newblock Simplicial localization of monoidal structures, and a non-linear
  version of {D}eligne's conjecture.
\newblock {\em Compos. Math.}, 141(1):253--261, 2005.

\bibitem{Lack_icons}
Stephen Lack.
\newblock Icons.
\newblock {\em Appl. Categ. Structures}, 18(3):289--307, 2010.

\bibitem{Lei2}
T.~{Leinster}.
\newblock {Homotopy Algebras for Operads}.
\newblock http://arxiv.org/abs/math/0002180.

\bibitem{Lei3}
T.~{Leinster}.
\newblock {Up-to-Homotopy Monoids}.
\newblock http://arxiv.org/abs/math/9912084.

\bibitem{Lurie_HTT}
Jacob Lurie.
\newblock {\em Higher topos theory}, volume 170 of {\em Annals of Mathematics
  Studies}.
\newblock Princeton University Press, Princeton, NJ, 2009.

\bibitem{Mac}
Saunders Mac~Lane.
\newblock {\em Categories for the working mathematician}, volume~5 of {\em
  Graduate Texts in Mathematics}.
\newblock Springer-Verlag, New York, second edition, 1998.

\bibitem{Shoikhet_Deligne}
B.~{Shoikhet}.
\newblock {Differential graded categories and Deligne conjecture}.
\newblock {\em ArXiv e-prints}, March 2013.

\bibitem{Shulman_monoidal_fib}
M.~A. {Shulman}.
\newblock {Framed bicategories and monoidal fibrations}.
\newblock {\em ArXiv e-prints}, June 2007.

\bibitem{Simpson_HTHC}
Carlos Simpson.
\newblock {\em Homotopy theory of higher categories}, volume~19 of {\em New
  Mathematical Monographs}.
\newblock Cambridge University Press, Cambridge, 2012.

\bibitem{Toen_Tannaka}
Bertrand To{\"e}n.
\newblock Dualit{\'e} de {Tannaka} sup{\'e}rieure {I}: Structure
  mono{\"\i}dales.
\newblock {\em Unpublished manuscript. Available on the author's website}, June
  2000.

\end{thebibliography}
\end{document}